\documentclass{article} 

\usepackage{amsmath, mathrsfs, amssymb, amsthm, mathtools, bbm} 

\usepackage{caption}
\usepackage{subcaption}

\usepackage{makecell} 
\usepackage{float} 

\usepackage[english]{babel} 
\usepackage[utf8]{inputenc} 
\usepackage[T1]{fontenc} 
\usepackage[margin=1.25in,footskip=0.35in]{geometry} 

\usepackage{tikz} 
\usetikzlibrary[topaths] 

\usepackage[symbol]{footmisc}

\usepackage{enumitem}

\usepackage{algorithm}
\usepackage{algcompatible}

\usepackage{hyperref} 

\newtheorem{theorem}{Theorem}

\newtheorem{lemma}{Lemma}

\newtheorem{corollary}{Corollary}
\newtheorem{proposition}{Proposition}

\newtheorem*{theorem*}{Theorem}
\newtheorem*{example*}{Example} 
\newtheorem*{definition*}{Definition}
\newtheorem*{lemma*}{Lemma}
\newtheorem*{remark*}{Remark}
\newtheorem*{corollary*}{Corollary}
\newtheorem*{proposition*}{Proposition}
\newtheorem*{assumption*}{Assumption}
\newtheorem*{claim*}{Claim}

\newtheoremstyle{TheoremNum}
        {\topsep}{\topsep}              
        {\itshape}                      
        {}                              
        {\bfseries}                     
        {.}                             
        { }                             
        {\thmname{#1}\thmnote{ \bfseries #3}}
\theoremstyle{TheoremNum}

\newtheoremstyle{LemmaNum}
        {\topsep}{\topsep}              
        {\itshape}                      
        {}                              
        {\bfseries}                     
        {.}                             
        { }                             
        {\thmname{#1}\thmnote{ \bfseries #3}}
\theoremstyle{LemmaNum}

\renewcommand{\Pr}{ \mathbb{P} }

\newcommand{\x}{ \mathbf{x} }

\renewcommand{\v}{ \mathbf{v} } 
\renewcommand{\u}{ \mathbf{u} } 

\renewcommand{\vec}{ \mathrm{\texttt{vec}} }

\newcommand{\e}{ \mathbf{e} }

\newcommand{\R}{\mathbb{R}}

\renewcommand{\S}{\mathcal{S}}
\newcommand{\E}{\mathbb{E}}

\renewcommand{\P}{\mathcal{P}}

\newcommand{\sign}{\mathrm{sign}}

\newcommand{\tr}{\mathrm{tr}}

\renewcommand{\[}{\left[ }
\renewcommand{\]}{\right] }

\newcommand{\<}{\left< }
\renewcommand{\>}{\right> }

\renewcommand{\(}{\left( }
\renewcommand{\)}{\right) }

\newcommand{\wt}{\widetilde }
\newcommand{\wh}{\widehat }

\def\submission{0}

\usepackage{natbib}

\begin{document} 

\title{Zeroth-order Low-rank Hessian Estimation via Matrix Recovery} 

\author{Tianyu Wang\footnote{wangtianyu@fudan.edu.cn} \quad Zicheng Wang\footnote{22110840011@m.fudan.edu.cn} \quad Jiajia Yu\footnote{jiajia.yu@duke.edu}} 

\date{} 

\maketitle 

\begin{abstract}

A zeroth-order Hessian estimator aims to recover the Hessian matrix of an objective function at any given point, using minimal finite-difference computations. This paper studies zeroth-order Hessian estimation for low-rank Hessians, from a matrix recovery perspective. Our challenge lies in the fact that traditional matrix recovery techniques are not directly suitable for our scenario. They either demand incoherence assumptions (or its variants), or require an impractical number of finite-difference computations in our setting. To overcome these hurdles, we employ zeroth-order Hessian estimations aligned with proper matrix measurements, and prove new recovery guarantees for these estimators. More specifically, we prove that for a Hessian matrix $H \in \mathbb{R}^{n \times n}$ of rank $r$, $ \mathcal{O}(nr^2 \log^2 n ) $ proper zeroth-order finite-difference computations ensures a highly probable exact recovery of $H$. Compared to existing methods, our method can greatly reduce the number of finite-difference computations, and does not require any incoherence assumptions.

\end{abstract}

\section{Introduction}


In machine learning, optimization and many other mathematical programming problems, the Hessian matrix plays an important role since it describes the landscape of the objective function. 
However, in many real-world scenarios, although we can access function values, the lack of analytic form for the objective function precludes direct Hessian computation. 
Therefore it is important to develop zeroth-order finite-difference Hessian estimators, i.e. to estimate the Hessian matrix by function evaluation and finite-difference.


Finite-difference Hessian estimation has a long history dating back to Newton's time. In recent years, the rise of large models and big data has posed the high-dimensionality of objective functions as a primary challenge in finite-difference Hessian estimation. To address this, stochastic Hessian estimators, like \citep{balasubramanian2021zeroth, wang2022hess, 10.1093/imaiai/iaad014, li2023stochastic}, have emerged to reduce the required number of function value samples. The efficiency of a Hessian estimator is measured by the \emph{sample complexity}, which quantifies the number of finite-difference computations needed.


Despite the high-dimensionality, the low-rank structure is prevalent in machine learning with high-dimensional datasets \citep{fefferman2016testing,doi:10.1137/18M1183480}. Numerous research directions, such as manifold learning \citep[e.g.,][]{ghojogh2023elements} and recommender systems \citep[e.g.,][]{resnick1997recommender}, actively leverage this low-rank structure. 
While there are many studies on stochastic Hessian estimators, as we detail in section \ref{sec:related-works}, none of them exploit the low-rank structure of the Hessian matrix. This omission can lead to overly conservative results and hinder the overall efficiency and effectiveness of the optimization or learning algorithms.
To fill in the gap, in this work, we develop an efficient finite-difference Hessian estimation method for low-rank Hessian via matrix recovery. While a substantial number of literature studies the sample complexity of low-rank matrix recovery, we emphasize that none of them are directly applicable to our scenario. This is either due to the overly restrictive global incoherence assumption or a prohibitively large number of finite-difference computations, as we discuss in detail in section \ref{sec:existing}. We develop a new method and prove that without the incoherence assumption, for an $n\times n$ Hessian matrix with rank $r$, we can exactly recover the matrix with high probability from $\mathcal{O}( nr^2 \log^2 n )$ proper zeroth-order finite-difference computations.

In the rest of this section, we present our problem formulation, discuss why existing matrix recovery methods fail on our problem and summarize our contribution.

\subsection{Hessian Estimation via Compressed Sensing Formulation}

To recover an $n \times n$ low-rank Hessian matrix $H$ using $\ll n^2 $ finite-difference operations, we use the following trace norm minimization approach \citep{fazel2002matrix,doi:10.1137/070697835,candes2010power,gross2011recovering,candes2012exact}: 
\begin{align} 
    \min_{\wh{H} \in \R^{n \times n} } \| \wh{H} \|_1 , 
    \quad 
    \text{ subject to } 
    \quad 
    \S \wh{H} = \S H , \label{eq:goal} 
\end{align} 
where $\S := \frac{1}{M} \sum_{i=1}^M \P_i$ and $\P_i$ is a matrix measurement operation that can be obtained via $\mathcal{O} (1)$ finite-difference computations. For our problem, it is worth emphasizing that $\P_i$ must satisfy the following requirements. 
\begin{itemize} 
    \item \textbf{(R1)} $\P_i$ is different from the sampling operation used for matrix completion. Otherwise an incoherence assumption is needed. See \textbf{(M1)} in Section \ref{sec:existing} for more details. 
    \item \textbf{(R2)} $\P_i$ cannot involve the inner product between the Hessian matrix and a general matrix, since this operation cannot be efficiently obtained through finite-difference computations. See \textbf{(M2)} in Section \ref{sec:existing} for more details. 
\end{itemize} 

Due to the above two requirements, existing theory for matrix recovery fails to provide satisfactory guarantees for low-rank Hessian estimation.

\subsection{Existing Matrix Recovery Methods} 
\label{sec:existing}

Existing methods for low-rank matrix recovery can be divided into two categories: matrix completion methods, and matrix recovery via linear measurements (or matrix regression type method). Unfortunately, both groups of methods are unsuitable for Hessian estimation tasks. 

\textbf{(M1) Matrix completion methods:} A candidate class of methods for low-rank Hessian estimation is matrix completion \citep{fazel2002matrix,cai2010singular,candes2010matrixnoise,candes2010power,5466511,lee2010admira,fornasier2011low,gross2011recovering,recht2011simpler,candes2012exact,hu2012fast,mohan2012iterative,negahban2012restricted,Wen2012,doi:10.1137/110845768,pmlr-v32-wanga14,chen2015incoherence,TANNER2016417,Gotoh2018,chen2020noisy,10375511}. 

The motivation for matrix completion tasks originated from the Netflix prize, where the challenge was to predict the ratings of all users on all movies based on only observing ratings of some users on some movies. In order to tackle such problems, it is necessary to assume that the nontrivial singular vectors of the matrix $H$ and the observation basis $\mathcal{B}$ are ``incoherent''. Incoherence \citep{candes2010power,gross2011recovering,candes2012exact,chen2015incoherence,negahban2012restricted}, or its alternatives \citep[e.g.,][]{negahban2012restricted}, implies that there is a sufficiently large angle between the singular vectors and the basis $\mathcal{B}$. The rationale behind this assumption can be explained as follows: Consider a matrix $H$ of size $n \times n$ with a one in its $(1,1)$ entry and zeros elsewhere. If we randomly observe a small fraction of the $n \times n$ entries, it is highly likely that we will miss the $(1,1)$ entry, making it difficult to fully recover the matrix. Therefore, an incoherence parameter $\nu$ is assumed between the given canonical basis $\mathcal{B}$ and the singular vectors of $H$, as illustrated in Figure \ref{fig:intro}. In the context of zeroth-order optimization, it is often necessary to recover the Hessian at any given point. However, assuming the Hessian is incoherence with the given basis over all points in the domain is overly restrictive. 

\textbf{(M2) Matrix recovery via linear measurements (matrix regression type recovery):} 
In the context of matrix recovery using linear measurements \citep{tan2011rank,ELDAR2012309,Chandrasekaran2012,RONG2021386}, we observe the inner product of the target matrix $H$ with a set of matrices $A_1, A_2, \cdots, A_M$. Specifically, we have the observation $ \< H, A_i \> := \tr ( H^* A_i ) $ and our goal is to recover $H$. In certain scenarios, there may be additional constraints on $A_i$ and the measurements might be corrupted by noise \citep{10.1214/10-AOS860,10.1214/20-AOS1980,doi:10.1080/01621459.2017.1389740}, which receives more attention from the statistics community. \cite{ELDAR2012309} proved that when the entries of $A_i$ are independently and identically distributed ($iid$) Gaussian, having $M \ge 4 nr - 4r^2 $ linear measurements ensures exact recovery of $H$. \cite{RONG2021386} showed that when the density of $(A_1, A_2, \cdots, A_M)$ is absolutely continuous, having $M > nr - r^2 $ measurements guarantees exact recovery of $H$. 

Despite the elegant results in matrix recovery using linear measurements, they are not applicable to Hessian estimation tasks. 
This limitation arises from the fact that a general linear measurement cannot be approximated by a zeroth-order estimation. 
To further illustrate this fact, let us consider the Taylor approximation, which, by the fundamental theorem of calculus, is the foundation for zeroth-order estimation. 
In the Taylor approximation of $f$ at $\x$, the Hessian matrix $\nabla^2 f (\x)$ will always appear as a bilinear form. Therefore, a linear measurement $\< A, \nabla^2 f (\x) \>$ for a general $A$ cannot be included in a Taylor approximation of $f$ at $\x$. 
In the language of optimization and numerical analysis, for a general measurement matrix $A$, one linear measurement $\< A ,H \>$ may require far more than $\mathcal{O}(1)$ finite-difference computations. Consequently, the theory providing guarantees for linear measurements does not extend to zeroth-order Hessian estimation. 







\subsection{Our Contribution}

In this paper, we introduce a low-rank Hessian estimation mechanism that simultaneously satisfies \textbf{(R1)} and \textbf{(R2)}. More specifically, 
\begin{itemize}
    \item We prove that, with a proper finite-difference scheme,
    $\mathcal{O} \( n r^2 \log^2 n  \)$ finite-difference computations are sufficient for guaranteeing an exact recovery of the Hessian matrix with high probability. Our approach simultaneously overcomes limitations of \textbf{(M1)} and \textbf{(M2)}. 
\end{itemize} 

In the realm of zeroth-order Hessian estimation, no prior arts provide high probability estimation guarantees for low-rank Hessian estimation tasks; See Section \ref{sec:related-works} for more discussions.

\subsection{Prior Arts on Hessian Estimation}  
\label{sec:related-works} 


Zeroth-order Hessian estimation dates back to the birth of calculus. In recent years, researchers from various fields have contributed to this topic \citep[e.g.,][]{broyden1973local,fletcher2000practical,spall2000adaptive,balasubramanian2021zeroth,li2023stochastic}.  




In quasi-Newton-type methods \citep[e.g.,][]{goldfarb1970family,shanno1970conditioning,broyden1973local,ren1983convergence,davidon1991variable,fletcher2000practical,spall2000adaptive,xu2001survey,Rodomanov2022}, gradient-based Hessian estimators were used for iterative optimization algorithms. 
Based on the Stein's identity \citep{10.1214/aos/1176345632}, \cite{balasubramanian2021zeroth} introduced a Stein-type Hessian estimator, and combined it with cubic regularized Newton's method \citep{nesterov2006cubic} for non-convex optimization. \cite{li2023stochastic} generalizes the Stein-type Hessian estimators to Riemannian manifolds. Parallel to \citep{balasubramanian2021zeroth,li2023stochastic}, \cite{wang2022hess,10.1093/imaiai/iaad014} investigated the Hessian estimator that inspires the current work. 

Yet prior to our work, no methods from the zeroth-order Hessian estimation community focuses on low-rank Hessian estimation.

\if\submission1
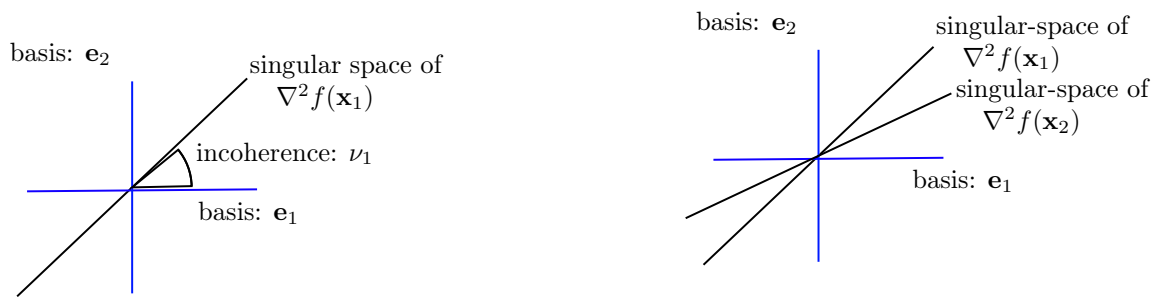
\begin{figure} 
    \begin{minipage}{.1\linewidth}
        \centering
        \tikzset{every picture/.style={line width=0.75pt}} 

\begin{tikzpicture}[x=0.75pt,y=0.75pt,yscale=-1,xscale=1]

\draw    (12.14,132.23) -- (128.14,22.23) ;
\draw [color={rgb, 255:red, 18; green, 0; blue, 255 }  ,draw opacity=1 ]   (17.14,79.23) -- (133.14,78.23) ;
\draw [color={rgb, 255:red, 0; green, 26; blue, 255 }  ,draw opacity=1 ]   (70.14,23.73) -- (70.14,130.73) ;
\draw  [draw opacity=0] (93.36,58.23) .. controls (97.48,63.26) and (99.99,69.64) .. (100.14,76.6) -- (70.14,77.23) -- cycle ; \draw   (93.36,58.23) .. controls (97.48,63.26) and (99.99,69.64) .. (100.14,76.6) ;  

\draw (128,9) node [anchor=north west][inner sep=0.75pt]   [align=left] {singular space of \\ \quad $\nabla^2 f (\x_1)$};
\draw (102,83) node [anchor=north west][inner sep=0.75pt]   [align=left] {basis: $\e_1$};
\draw (7,3) node [anchor=north west][inner sep=0.75pt]   [align=left] {basis: $\e_2$};
\draw (101,53) node [anchor=north west][inner sep=0.75pt]   [align=left] {incoherence: $\nu_1$};

\end{tikzpicture} 
    \end{minipage} 
    \hspace{4.2cm}
    \begin{minipage}{.1\linewidth}
        \centering
        \tikzset{every picture/.style={line width=0.75pt}} 

\begin{tikzpicture}[x=0.75pt,y=0.75pt,yscale=-1,xscale=1]

\draw    (14.14,131.23) -- (130.14,21.23) ;
\draw [color={rgb, 255:red, 18; green, 0; blue, 255 }  ,draw opacity=1 ]   (19.14,78.23) -- (135.14,77.23) ;
\draw [color={rgb, 255:red, 0; green, 26; blue, 255 }  ,draw opacity=1 ]   (72.14,22.73) -- (72.14,129.73) ;
\draw    (5.14,107.73) -- (139.14,44.73) ;

\draw (130,4) node [anchor=north west][inner sep=0.75pt]   [align=left] {singular-space of\\ \quad$\nabla^2 f (\x_1)$};
\draw (118,82) node [anchor=north west][inner sep=0.75pt]   [align=left] {basis: $\e_1$};
\draw (9,2) node [anchor=north west][inner sep=0.75pt]   [align=left] {basis: $\e_2$};
\draw (140,36) node [anchor=north west][inner sep=0.75pt]   [align=left] {singular-space of\\ \quad$\nabla^2 f (\x_2)$};

\end{tikzpicture} 
    \end{minipage} 
    \caption{Incoherence condition for $\nabla^2 f (\x)$ at multiple points. When the Hessian of $f$ is low-rank or approximately low-rank, a matrix completion guarantee for $\nabla^2 f (\x)$ at all $\x$ requires an incoherence condition to hold uniformly over $\x$. As illustrated in the right subfigure, such requirement is overly restrictive. } 
    \label{fig:intro} 
\end{figure} 
\fi 

\if\submission0 
\begin{figure} 
    \begin{centering}
    \begin{subfigure}[b]{0.4\textwidth}
         \tikzset{every picture/.style={line width=0.75pt}} 

\begin{tikzpicture}[x=0.75pt,y=0.75pt,yscale=-1,xscale=1]

\draw    (12.14,132.23) -- (128.14,22.23) ;
\draw [color={rgb, 255:red, 18; green, 0; blue, 255 }  ,draw opacity=1 ]   (17.14,79.23) -- (133.14,78.23) ;
\draw [color={rgb, 255:red, 0; green, 26; blue, 255 }  ,draw opacity=1 ]   (70.14,23.73) -- (70.14,130.73) ;
\draw  [draw opacity=0] (93.36,58.23) .. controls (97.48,63.26) and (99.99,69.64) .. (100.14,76.6) -- (70.14,77.23) -- cycle ; \draw   (93.36,58.23) .. controls (97.48,63.26) and (99.99,69.64) .. (100.14,76.6) ;  

\draw (128,9) node [anchor=north west][inner sep=0.75pt]   [align=left] {singular space of \\ \quad $\nabla^2 f (\x_1)$};
\draw (102,83) node [anchor=north west][inner sep=0.75pt]   [align=left] {basis: $\e_1$};
\draw (7,3) node [anchor=north west][inner sep=0.75pt]   [align=left] {basis: $\e_2$};
\draw (101,53) node [anchor=north west][inner sep=0.75pt]   [align=left] {incoherence: $\nu_1$};

\end{tikzpicture} 
     \end{subfigure} 
     \hfill 
     \begin{subfigure}[b]{0.4\textwidth}
         \tikzset{every picture/.style={line width=0.75pt}} 

\begin{tikzpicture}[x=0.75pt,y=0.75pt,yscale=-1,xscale=1]

\draw    (14.14,131.23) -- (130.14,21.23) ;
\draw [color={rgb, 255:red, 18; green, 0; blue, 255 }  ,draw opacity=1 ]   (19.14,78.23) -- (135.14,77.23) ;
\draw [color={rgb, 255:red, 0; green, 26; blue, 255 }  ,draw opacity=1 ]   (72.14,22.73) -- (72.14,129.73) ;
\draw    (5.14,107.73) -- (139.14,44.73) ;

\draw (130,4) node [anchor=north west][inner sep=0.75pt]   [align=left] {singular-space of\\ \quad$\nabla^2 f (\x_1)$};
\draw (118,82) node [anchor=north west][inner sep=0.75pt]   [align=left] {basis: $\e_1$};
\draw (9,2) node [anchor=north west][inner sep=0.75pt]   [align=left] {basis: $\e_2$};
\draw (140,36) node [anchor=north west][inner sep=0.75pt]   [align=left] {singular-space of\\ \quad$\nabla^2 f (\x_2)$};

\end{tikzpicture} 
     \end{subfigure}
    \end{centering}
    \caption{Incoherence condition for $\nabla^2 f (\x)$ at multiple points. When the Hessian of $f$ is low-rank or approximately low-rank, a matrix completion guarantee for $\nabla^2 f (\x)$ at all $\x$ requires an incoherence condition to hold uniformly over $\x$. As illustrated in the right subfigure, such requirement is overly restrictive. } 
    \label{fig:intro} 
\end{figure} 
\fi

\section{Notations and Conventions}

\label{sec:not}

Before proceeding to main results, we lay out some conventions and notations that will be used throughout the paper. We use the following notations for matrix norms: 
\begin{itemize} 
    \item $\| \cdot \|$ is the operator norm (Schatten $\infty$-norm); 
    \item $\| \cdot \|_2$ is the Euclidean norm (Schatten $2$-norm); 
    \item $\| \cdot \|_1$ is the trace norm (Schatten $1$-norm).  
\end{itemize} 
Also, the notation $\| \cdot \|$ is overloaded for vector norm and tensor norm. For a vector $ \v \in \R^n $, $\| \cdot \|$ is its Euclidean norm; For a tensor $ V \in \( \R^n \)^{\otimes p} $ ($p \ge 2$), $\| \cdot \|$ is its Schatten $\infty$-norm. 
For any matrix $A$ with singular value decomposition $ A = U \Sigma V^\top $, we define $\sign(A) = U \sign (\Sigma) V^\top $ where $ \sign (\Sigma) $ applies a $\sign$ function to each entry of $\Sigma$. 

For a vector $\u = \( u_1,u_2,\cdots, u_n \)^\top \in \R^n$ and a positive number $r \le n$, we define notations 
\begin{align*}
    \u_{:r} = \( u_1, u_2, \cdots, u_r, 0,0,\cdots, 0 \)^\top \; \text{and} \;  \u_{r:} = \(  0,0,\cdots, 0 , u_r, u_{r+1}, \cdots, u_n \)^\top . 
\end{align*}


Also, we use $C$ and $c$ to denote unimportant absolute constants that does not depend on $n$ or $r$. The numbers $C$ and $c$ may or may not take the same value at each occurrence. 





\section{Main Results} 

\label{sec:main}





We start with a finite-difference scheme that can be viewed as a matrix measurement operation. 
The Hessian of a function $f : \R^n \to \R$ at a given point $\x$ can be estimated as follows \citep{wang2022hess,10.1093/imaiai/iaad014} 
\begin{align}
    &\wh{\nabla}^2 f (\x) 
    := \nonumber \\
    &n^2 \frac{ f (\x + \delta \v + \delta \u ) - f (\x - \delta \v + \delta \u ) - f (\x + \delta \v - \delta \u ) + f (\x - \delta \v - \delta \u ) }{ 4 \delta^2 } \u \v^\top , \label{eq:hess-est}
\end{align} 
where $\delta$ is the finite-difference granularity, and $ \u, \v $ are finite-difference directions. Difference choices of laws of $\u$ and $\v$ leads to different Hessian estimators. For example, $ \u, \v $ can be independent vectors uniformly distributed over the canonical basis $\{ \e_1, \e_2, \cdots, \e_n \}$. 

We start our discussion by showing that the Hessian estimator (\ref{eq:hess-est}) can indeed be viewed as a matrix measurement. 

\begin{proposition} 
    \label{prop:measure}
    Consider an estimator defined in (\ref{eq:hess-est}). Let the underlying function $f$ be twice continuously differentiable. Let $\u,\v$ be two random vectors such that $ \| \u \| , \| \v \| < \infty $ $a.s$. Then for any fixed $\x \in \R^n$, 
    \begin{align*}
        \wh{\nabla}^2 f (\x) \to_d n^2 \u \u^\top \nabla^2 f (\x) \v \v^\top
    \end{align*}
    as $\delta \to 0_+$, where $\to_d$ denotes convergence in distribution. 
\end{proposition} 

\begin{proof} 
    
    By Taylor's Theorem (with integrable remainder) and that the Hessian matrix is symmetric, we have 
    \begin{align*} 
        \wh{\nabla}^2 f (\x) 
        =& \; 
        \frac{n^2}{4} 
        \(  \( \v + \u \)^\top \nabla^2 f (\x) \( \v + \u \) - \( \v - \u \)^\top \nabla^2 f (\x) \( \v - \u \) \) \u \v^\top \\
        &+ \mathcal{O} \( \delta \( \| \v \| + \| \u \| \)^3 \) \\ 
        =& \; 
        n^2 \u^\top \nabla^2 f (\x) \v \u \v^\top + \mathcal{O} \( \delta \( \| \v \| + \| \u \| \)^3 \) \\
        =&\; 
        n^2 \u  \u^\top \nabla^2 f (\x) \v \v^\top + \mathcal{O} \( \delta \( \| \v \| + \| \u \| \)^3 \) . 
    \end{align*} 
    
    As $\delta \to 0_+$, the estimator (\ref{eq:hess-est}) converges to $ n^2 \u \u^\top \nabla^2 f (\x) \v \v^\top $ in distribution. 

\end{proof}

With Proposition \ref{prop:measure} in place, we see that matrix measurements of the form 
\begin{align*} 
    \P: H \mapsto n^2 \u \u^\top H \v \v^\top 
\end{align*} 
for some $\u,\v$ can be efficiently computed via finite-difference computations. For the convex program (\ref{eq:goal}) with sampling operators taking the above form, we have the following guarantee. 






\begin{theorem} 
    \label{thm:main} 
    Consider the problem (\ref{eq:goal}). Let the sampler $ \S = \frac{1}{M} \sum_{i=1}^M \P_i $ be constructed with $ \P_i : A \mapsto n^2 \u_i\u_i^\top A \v_i \v_i^\top $ and $ \u_i, \v_i \overset{iid}{\sim} \text{Unif} ( \mathbb{S}^{n-1}) $. 
    Then there exists an absolute constant $C$, such that if the number of samples $M \ge C\cdot n r^2 \log^2 (n) $ where $r := rank (H)$, then 
    with probability larger than $1 - \frac{1}{n}$, the solution to (\ref{eq:goal}), denoted by $\wh{H}$, satisfies $ \wh{H} = H $. 
\end{theorem} 


As a direct consequence of Theorem \ref{thm:main}, we have the following result. 
\begin{corollary} 
    \label{cor:hess}
    Let the finite-difference granularity $\delta > 0$ be small. Let $\x \in \R^n$ and let $f$ be twice continuously differentiable. 
    Suppose there exists $ H $ with $rank (H) = r$ such that $ \| H - \nabla^2 f (\x) \| \le \epsilon$ for some $\epsilon \ge 0$, and the estimator (\ref{eq:hess-est}) with $ \u , \v \overset{iid}{\sim} \mathrm{Unif} (\mathbb{S}^{n-1}) $ satisfies 
    \begin{align*} 
        &\;\frac{ f (\x + \delta \v + \delta \u ) - f (\x - \delta \v + \delta \u ) - f (\x + \delta \v - \delta \u ) + f (\x - \delta \v - \delta \u ) }{ 4 \delta^2 } \u \v^\top \\
        =_d &\;
        \u \u^\top H \v \v^\top ,
    \end{align*} 
    where $ =_d $ denotes distributional equivalence. 
    There exists an absolute constant $C$, such that if more than $C \cdot nr^2 \log^2 n $ zeroth-order finite-difference are obtained, then with probability exceeding $ 1 - \frac{1}{n} $, the solution $\wh{H}$ to (\ref{eq:goal}) satisfies $ \| \wh{H} - \nabla^2 f (\x) \| \le \epsilon $. 
\end{corollary} 

By Proposition \ref{prop:measure}, we know as $\delta \to 0^+$, 
\begin{align*}
    \frac{ f (\x + \delta \v + \delta \u ) - f (\x - \delta \v + \delta \u ) - f (\x + \delta \v - \delta \u ) + f (\x - \delta \v - \delta \u ) }{ 4 \delta^2 } \u \v^\top
\end{align*} 
converges to $ \u \u^\top \nabla^2 f (\x) \v \v^\top $ in distribution. Therefore, Corollary \ref{cor:hess} implies that the estimator (\ref{eq:hess-est}) together with a convex program (\ref{eq:goal}) provides a sample efficient low-rank Hessian estimator. Corollary \ref{cor:hess} also implies a guarantee for approximately low-rank Hessian. 

The rest of this section is devoted to proving Theorem \ref{thm:main} and thus also Corollary \ref{cor:hess}. 

\subsection{Preparations}

To describe the recovering argument for a symmetric low-rank matrix $H \in \R^{n\times n}$ with $ rank (H) = r $, we consider the eigenvalue decomposition of $H = U \Lambda U^\top$ ($U \in \R^{n \times r}$ and $\Lambda \in \R^{r \times r}$), and a subspace of $\R^{n\times n}$ defined by 
\begin{align*} 
    T := \{ A \in \R^{n \times n} : \( I - P_U \) A \( I - P_U \) = 0 \} , 
\end{align*}  
where $P_U$ is the projection onto the columns of $U$. 
We also define a projection operation onto $T$:  
\begin{align*} 
    \P_T : A \mapsto P_U A + A P_U - P_U A P_U . 
\end{align*} 



Let $\wh{H}$ be the solution of (\ref{eq:goal}) and let $\Delta := \wh{H} - H$. We start with the following lemma, which can be extracted from matrix completion literature \citep[e.g.,][]{candes2010power,gross2011recovering,candes2012exact}. 

\begin{lemma} 
    \label{lem:prepare} 
    Let $ \wh{H} $ be the solution of the program (\ref{eq:goal}) and let $ \Delta := \wh{H} - H $. 
    Then it holds that 
    \begin{align} 
        \< \sign(H),  P_{U} \Delta P_{U} \> + \| \Delta_{T}^\perp \|_1 
        \le 0 , \label{eq:need-cert} 
    \end{align} 
    where $ \Delta_{T}^\perp := \P_{T}^\perp \Delta $. 
\end{lemma} 

\begin{proof}


    Since $H \in T$, we have 
    \begin{align} 
        & \; \| H + \Delta \|_1 
        \ge 
        \| P_{U} ( H + \Delta ) P_{U} \|_1 + \| P_{U}^\perp ( H + \Delta ) P_{U}^\perp \|_1 \\
        =& \;  
        \| H + P_{U} \Delta P_{U} \|_1 + \| \Delta_{T}^\perp \|_1 , \label{eq:lem1-1}
    \end{align}
    where the first inequality uses the ``pinching'' inequality (Exercise II.5.4 \& II.5.5 in \citep{bhatia1997matrix}).  

    Since $ \| \sign (H) \| = 1 $, we continue the above computation, and get 
    \begin{align}
        (\ref{eq:lem1-1})
        =& \; 
        \| \sign (H) \| \| H + P_{U} \Delta P_{U} \|_1 + \| \Delta_{T}^\perp \|_1 \nonumber \\ 
        \ge & \; 
        \< \sign (H) , H + P_{U} \Delta P_{U} \> + \| \Delta_{T}^\perp \|_1 \nonumber \\ 
        =& \; 
        \| H \|_1 + \< \sign(H),  P_{U} \Delta P_{U} \> + \| \Delta_{T}^\perp \|_1. \label{eq:lem1-2} 
    \end{align} 
    On the second line, we use the H\"older's inequality. On the third line, we use that $ \| A \|_1 = \< \sign (A) , A \> $ for any real matrix $A$.
    
    Since $ \wh{H} $ solves (\ref{eq:goal}), we know $ \| H \|_1 \ge \| \wh{H} \|_1 = \| H + \Delta \|_1 $. Thus rearranging terms in (\ref{eq:lem1-2}) finishes the proof. 
\end{proof}

\subsection{The High Level Roadmap} 

With estimator (\ref{eq:hess-est}) and Lemma \ref{lem:prepare} in place, we are ready to present the high-level roadmap of our argument. 
On a high level, the rest of the paper aims to prove the following two arguments: 
\begin{itemize} 
    \item \textbf{(A1):} With high probability, $ \| \Delta_T \|_2 \le 2n \| \Delta_T^\perp \|_2 $, where $\Delta_T := \P_{T} \Delta $. 
    \item \textbf{(A2):} With high probability, $ \< \sign(H),  P_{U} \Delta P_{U} \> \ge - \frac{1}{n^{20}} \| \Delta_T \|_1 - \frac{1}{2} \| \Delta_T^\perp \|_1 $, where $\Delta_T^\perp := \Delta - \Delta_T$. 
\end{itemize} 

Once \textbf{(A1)} and \textbf{(A2)} are in place, we can quickly prove Theorem \ref{thm:main}. 

\begin{proof}[Sketch of proof of Theorem \ref{thm:main} with \textbf{(A1)} and \textbf{(A2)} assumed] 
    
    Now, by Lemma \ref{lem:prepare} and \textbf{(A1)}, we have, with high probability, 
    \begin{align*} 
        0 
        \overset{\text{by Lemma \ref{lem:prepare}}}{\ge}& \;  
        \< \sign(H),  P_{U} \Delta P_{U} \> + \| \Delta_{T}^\perp \|_1 \\ 
        \overset{\text{by \textbf{(A2)}}}{\ge}& \;  
        \frac{1}{2} \| \Delta_T^\perp \|_1 - \frac{1}{n^{20}} \| \Delta_T \|_1 \\ 
        \overset{\text{by \textbf{(A1)}}}{\ge}& \;  
        \frac{1}{2} \| \Delta_T^\perp \|_1 - \frac{2}{n^{18}} \| \Delta_T^\perp \|_1 , 
    \end{align*} 
    which implies $ \| \Delta_T^\perp \|_1 = 0 $ \emph{w.h.p}. Finally another use of \textbf{(A1)} implies $ \| \Delta \|_1 = 0 $ \emph{w.h.p.}, which concludes the proof. 
\end{proof} 

Therefore, the core argument reduces to proving \textbf{(A1)} and \textbf{(A2)}. 
In the next subsection, we prove \textbf{(A1)} and \textbf{(A2)} for the random measurements obtained by the Hessian estimator (\ref{eq:hess-est}), without any incoherence-type assumptions. 






\subsection{The Concentration Arguments } 
\label{sec:mat-rec} 


For the concentration argument, we need to make several observations. One of the key observations is that the spherical measurements are rotation-invariant and reflection-invariant. More specifically, for the random measurement $ \P H = n^2 \u \u^\top H \v \v^\top  $ with $\u, \v \overset{iid}{\sim} \mathrm{Unif} (\mathbb{S}^{n-1})$, we have
\begin{align*}
    n^2 \u \u^\top H \v \v^\top =_d n^2 Q \u \u^\top Q^\top H Q \v \v^\top Q^\top
\end{align*}
for any orthogonal matrix $Q$, where $=_d$ denotes distributional equivalence. With a properly chosen $Q$, we have 
\begin{align*}
    n^2 \u \u^\top H \v \v^\top =_d n^2 Q \u \u^\top \Lambda \v \v^\top Q^\top , 
\end{align*}
where $\Lambda$ is the diagonal matrix consisting of eigenvalues of $H$. This observation makes calculating the moments of $ \P H $ possible. With the moments of the random matrices properly controlled, we can use matrix-valued Cramer--Chernoff method to arrive at the matrix concentration inequalities. 

Another useful property is the Kronecker product and the vectorization of the matrices. Let $ \vec \(\cdot\) $ be the vectorization operation of a matrix. Then as per how $\P_{T}$ is defined, we have, for any $A \in \R^{n \times n}$, 
\begin{align} 
    \vec \( \P_{T} A \) 
    =& \;  
    \vec \( P_{U} A + A P_{U} - P_{U} A P_{U} \) \nonumber \\ 
    =& \;  
    \( P_{U} \otimes I_n + I_n \otimes P_{U} - P_{U} \otimes P_{U} \) \vec \( A \). \label{eq:represent} 
\end{align} 
The above formula implies that $ \P_T $ can be represented as a matrix of size $n^2 \times n^2$. Similarly, the measurement operators $\P : A \mapsto n^2 \u \u^\top A \v \v^\top$ can also be represented as a matrix of size $n^2 \times n^2$. 
Compared to the matrix completion problem, the importance of vectorization presentation and Kronecker product is more pronounced for our case. The reason is again the absence of an incoherence-type assumption. More specifically, a vectorized representation is useful in controlling the cumulant generating function of the random matrices associated with the spherical measurements. 

Finally some additional care is needed to properly control the high moments of $\P H$. Such additional care is showcased in an inequality stated below in Lemma \ref{lem:wu}. An easy upper bound for the LHS of (\ref{eq:wu}) is $ \mathcal{O} (r^p) $. However, an $ \mathcal{O} (r^p) $ bound for the LHS of (\ref{eq:wu}) will eventually result in a loss in a factor of $r$ in the final bound. Overall, tight control is needed over several different places, in order to get the final recovery bound in Theorem \ref{thm:main}. 

\begin{lemma}
    \label{lem:wu}
    Let $ r $ and $p \ge 2$ be positive integers. Then it holds that 
    \begin{align} 
        \max_{ \alpha_1, \alpha_2, \cdots, \alpha_r \ge 0; \; \sum_{i=1}^r \alpha_i = 2p ; \; \alpha_i \text{ even} } \frac{(2p)!}{p!} \prod_{i=1}^r \frac{(\frac{\alpha_i}{2})!}{\alpha_i!} 
        \le 
        (100 r)^{p-1} . \label{eq:wu}
    \end{align} 
\end{lemma} 

\begin{proof}
    \textbf{Case I: } $ r \le \frac{1}{2} 50^{p-1} $. 
    Note that 
    \begin{align}
        \frac{(\frac{\alpha_i}{2})!}{\alpha_i!} 
        \le 
        \frac{1}{(\frac{\alpha_i}{2})^{(\frac{\alpha_i}{2})}} 
        \quad \text{ and thus } \quad 
        \log  \frac{(\frac{\alpha_i}{2})!}{\alpha_i!} 
        \le 
        - \frac{\alpha_i}{2} \log (\frac{\alpha_i}{2}) . \label{eq:tmp1}
    \end{align}
    Since the function $ x \mapsto - x \log x $ is concave, Jensen's inequality gives  
    \begin{align} 
        \frac{ - \sum_{i=1}^r \frac{\alpha_i}{2} \log (\frac{\alpha_i}{2}) }{ r } 
        \le 
        - \frac{ \frac{\sum_{i=1}^r \alpha_i }{r} }{2} \log \(  \frac{ \frac{\sum_{i=1}^r \alpha_i }{r} }{2} \)
        = 
        - \frac{p}{r} \log \frac{p}{r} . \label{eq:tmp2}
    \end{align} 
    Combining (\ref{eq:tmp1}) and (\ref{eq:tmp2}) gives 
    \begin{align*} 
        \log \prod_{i=1}^r \frac{(\frac{\alpha_i}{2})!}{\alpha_i!} 
        \le 
        - \sum_{i=1}^r \frac{\alpha_i}{2} \log (\frac{\alpha_i}{2}) 
        \le 
        - p \log \frac{p}{r} , 
    \end{align*} 
    which implies 
    \begin{align*} 
        \frac{(2p)!}{p!} \prod_{i=1}^r \frac{(\frac{\alpha_i}{2})!}{\alpha_i!} 
        \le 
        (2p)^p (\frac{r}{p})^p 
        \le 
        (2r)^p \le (100r)^{p-1},  
    \end{align*} 
    where the last inequality uses $ r \le \frac{1}{2} 50^{p-1}  $. 

    \textbf{Case II: } $ r > \frac{1}{2} 50^{p-1} $. 
    For this case, we first show that the maximum of $ \prod_{i=1}^r \frac{(\frac{\alpha_i}{2})!}{\alpha_i!}  $ is obtained when $ | \alpha_i - \alpha_j | \le 2 $ for all $i,j$. To show this, let there exist $\alpha_k$ and $\alpha_j$ such that $ |\alpha_k - \alpha_j| > 2$. Without loss of generality, let $\alpha_k > \alpha_j + 2$. Then 
    \begin{align*} 
        \frac{(\frac{\alpha_k}{2})!}{\alpha_k!}  \cdot \frac{(\frac{\alpha_j}{2})!}{\alpha_j!} 
        \le 
        \frac{(\frac{\alpha_k-2}{2})!}{(\alpha_k-2)!}  \cdot \frac{(\frac{(\alpha_j+2)}{2})!}{(\alpha_j+2)!} .  
    \end{align*} 
    Therefore, we can increase the value of $ \prod_{i=1}^r \frac{(\frac{\alpha_i}{2})!}{\alpha_i!}  $ until $ | \alpha_i - \alpha_j | \le 2 $ for all $i,j$. By the above argument, we have, for $ r > \frac{1}{2} 50^{p-1} \ge p $, 
    \begin{align*} 
        \max_{ \alpha_1, \alpha_2, \cdots, \alpha_r \ge 0; \; \sum_{i=1}^r \alpha_i = 2p ; \; \alpha_i \text{ even} } \prod_{i=1}^r \frac{(\frac{\alpha_i}{2})!}{\alpha_i!} 
        \le 
        \( \frac{1}{2}\)^{p} \cdot \( \frac{0!}{0!} \)^{r-p} 
        = 
        \frac{1}{2^p } . 
    \end{align*}  
    
    Therefore, we have 
    \begin{align*} 
        & \; \max_{ \alpha_1, \alpha_2, \cdots, \alpha_r \ge 0; \; \sum_{i=1}^r \alpha_i = 2p ; \; \alpha_i \text{ even} } \frac{(2p)!}{p!} \prod_{i=1}^r \frac{(\frac{\alpha_i}{2})!}{\alpha_i!} 
        \le 
        (2p)^p \cdot 2^{-p} \\ 
        =& \;  
        p^p 
        \le 
        (50 \cdot 50^{p-1})^{p-1} 
        \le 
        (100 r)^{p-1} . 
    \end{align*} 
    
\end{proof}







With all the above preparation in place, we next present Lemma \ref{lem:op-concen}, which is the key step leading to \textbf{(A1)}. 

\begin{lemma} 
    \label{lem:op-concen} 
    Let 
    \begin{align*} 
        \mathcal{E}_1 := \left\{ \left\| \P_{T} \S \P_{T} - \P_{T} \right\| \le \frac{1}{4} \right\} , 
    \end{align*} 
    where $ \P_{T} $ and $ \S $ are regarded as matrices of size $n^2 \times n^2$. Pick any $\delta \in (0,1)$. 
    Then there exists some constant $C$, such that when $ M \ge C n r \log (1/\delta) $, it holds that $ \Pr \( \mathcal{E}_1 \) \ge 1 - \delta $. 
\end{lemma}

The operators $\P_T$ and $\S$ can be represented as matrix of size $n^2 \times n^2$. Therefore, we can apply matrix-valued Cramer--Chernoff-type argument (or matrix Laplace argument \citep{LIEB1973267}) to derive a concentration bound. In \citep{tropp2012user,tropp2015introduction}, a master matrix concentration inequality is presented. This result is stated below in Theorem \ref{thm:tropp}.

\begin{theorem}[\cite{tropp2015introduction}] 
    \label{thm:tropp}
    Consider a finite sequence $\{X_k \}$ of independent, random, Hermitian matrices of the same size. Then for all $t \in \R$,
    \begin{align*}
        \Pr \( \lambda_{\max} \( \sum_k X_k \) \ge t \) 
        \le \inf_{\theta >0 } e^{- \theta t} \tr \exp \( \sum_k \log \E e^{\theta X_k} \) , 
    \end{align*}
    and 
    \begin{align*}
        \Pr \( \lambda_{\min} \( \sum_k X_k \) \le t \) 
        \le \inf_{\theta <0 } e^{- \theta t} \tr \exp \( \sum_k \log \E e^{\theta X_k} \) . 
    \end{align*}

\end{theorem}

For our purpose, a more convenient form is the matrix concentration inequality with Bernstein's conditions on the moments. Such results may be viewed as corollaries to Theorem \ref{thm:tropp}, and a version is stated below in Theorem \ref{thm:zhu}. 

\begin{theorem}[\cite{zhu2012short,6905847}] 
    \label{thm:zhu} 
    If a finite sequence $\{X_k : k = 1 ,\cdots, K \}$ of independent, random, self-adjoint matrices with dimension $n$, all of which satisfy the Bernstein’s moment condition, i.e.
    \begin{align*}
        \E \[ X_k^p \] 
        \preceq 
        \frac{p!}{2} B^{p-2} \Sigma_2, \quad \text{ for } p \ge 2, 
    \end{align*}
    where $B$ is a positive constant and $\Sigma_2$ is a positive semi-definite matrix, then,
    \begin{align*}
        \Pr \( \lambda_1 \(\sum_k X_k \) \ge \lambda_1 \( \sum_k \E X_k \) + \sqrt{ 2 K \theta \lambda_1 \(\Sigma_2 \) } + \theta B \) 
        \le 
        n \exp \( - \theta \) , 
    \end{align*} 
    for each $\theta > 0$. 
\end{theorem} 


Another useful property is the moments of spherical random variables, stated below in Proposition \ref{prop:p}. The proof of Proposition \ref{prop:p} is in the Appendix. 

\begin{proposition}
    \label{prop:p}
    Let $\v$ be uniformly sampled from $\mathbb{S}^{n-1}$ ($n \ge 2$). It holds that 
    \begin{align*}
        \E \[ v_i^p \] = \frac{ (p-1) (p-3) \cdots 1 }{ n (n+2) \cdots (n+p - 2) } 
    \end{align*}
    for all $i = 1,2,\cdots,n$ and any positive even integer $p$. 
\end{proposition}

With the above results in place, we can now prove Lemma \ref{lem:op-concen}. 
\begin{proof}[Proof of Lemma \ref{lem:op-concen}] 
    {Fix $\delta \in (0,1)$, and let $ M > Cnr\log(1/\delta) $ for some absolute constant $C$. }
    Following the similar reasoning for (\ref{eq:represent}), we can represent $\P$ as 
    \begin{align}
        \P = n^2 \u \u^\top \otimes \v \v^\top , \label{eq:def-P}
    \end{align}
    where $\u,\v \overset{iid}{\sim} \mathrm{Unif} (\mathbb{S}^{n-1})$. 
    
    Thus, by viewing $ \P $ and $\P_{T} $ as matrices of size $n^2 \times n^2$, we have 
    \begin{align*} 
        \P_{T} \P \P_{T} 
        =& \;  
        n^2 \( P_{U} \otimes I_n + I_n \otimes P_{U} - P_{U} \otimes P_{U} \) \( \u \u^\top \otimes \v \v^\top \) \\ 
        & \cdot \( P_{U} \otimes I_n + I_n \otimes P_{U} - P_{U} \otimes P_{U} \) . 
    \end{align*}


    Let $ Q $ be an orthogonal matrix such that 
    \begin{align*}
        Q P_{U} Q^\top = I_n^{:r} 
        := 
        \begin{bmatrix}
        I & 0_{r \times (n-r)} \\ 
        0_{ (n-r) \times r} & 0_{(n-r) \times (n-r)} . 
        \end{bmatrix}
    \end{align*}
    
    Since the distributions of $ \u $ and $ \v $ are rotation-invariant and reflection-invariant, we know  
    \begin{align} 
        & \; \( I_n^{:r} \otimes I_n + I_n \otimes I_n^{:r} - I_n^{:r} \otimes I_n^{:r}  \) \P \( I_n^{:r} \otimes I_n + I_n \otimes I_n^{:r} - I_n^{:r} \otimes I_n^{:r}  \) \nonumber \\ 
        =& \;  
        \( Q \otimes Q \) \P_{T} \( Q^\top \otimes Q^\top \) \P \( Q \otimes Q \) \P_{T} \( Q^\top \otimes Q^\top \) \nonumber \\ 
        =_d& \;   
        \( Q \otimes Q \) \P_{T} \P \P_{T} \( Q^\top \otimes Q^\top \) , \label{eq:equi-dist} 
    \end{align} 
    where $=_d$ denotes distributional equivalence. 
    
    Therefore, it suffices to study the distribution of 
    \begin{align*} 
        \( I_n^{:r} \otimes I_n + I_n \otimes I_n^{:r} - I_n^{:r} \otimes I_n^{:r}  \) \P_i \( I_n^{:r} \otimes I_n + I_n \otimes I_n^{:r} - I_n^{:r} \otimes I_n^{:r}  \) . 
    \end{align*} 
        
    For simplicity, introduce notation 
    \begin{align*} 
        \mathcal{R}_{T} 
        := 
        I_n^{:r} \otimes I_n + I_n \otimes I_n^{:r} - I_n^{:r} \otimes I_n^{:r} 
        = 
        I_n^{:r} \otimes I_n + I_n^{r+1:} \otimes I_n^{:r} , 
    \end{align*} 
    and we have 
    \begin{align*} 
        \mathcal{R}_{T} \P \mathcal{R}_{T} 
        =& \;  
        n^2 \u_{:r} \u_{:r}^\top \otimes \v \v^\top + n^2 \u_{r+1:} \u_{r+1:}^\top \otimes \v_{:r} \v_{:r}^\top \\
        &+ n^2 \u_{r+1:} \u_{:r}^\top \otimes \v_{:r} \v^\top + n^2 \u_{:r} \u_{r+1:}^\top \otimes \v \v_{:r}^\top 
    \end{align*}

    For simplicity, introduce 
    \begin{align*}
        & \; X := n^2 \u_{:r} \u_{:r}^\top \otimes \v \v^\top 
        \quad \\ 
        & \; Y := n^2 \u_{r+1:} \u_{r+1:}^\top \otimes \v_{:r} \v_{:r}^\top \\ 
        & \;
        Z := 
        n^2 \u_{r+1:} \u_{:r}^\top \otimes \v_{:r} \v^\top + n^2 \u_{:r} \u_{r+1:}^\top \otimes \v \v_{:r}^\top . 
    \end{align*}

    Next we will show that average of $ iid $ copies of $X$, $Y$, $Z$ concentrates to $ \E X $, $\E Y$, $\E Z$ respectively. To do this, we bound the moments of $X$, $Y$ and $Z$, and apply Theorem \ref{thm:zhu}. 

    \textbf{Bounding $X$ and $Y$. }
    The second moment of $X$ is 
    \begin{align*} 
        \E \[ X^2 \] 
        = 
        n^4 \E \[ \( \u_{:r}^\top \u_{:r} \) \u_{:r} \u_{:r}^\top \otimes \v \v^\top \] 
        \preceq 
        3nr , 
    \end{align*}
    where the last inequality follows from Proposition \ref{prop:p}. Thus the centralized second moment of $X$ is bounded by 
    \begin{align*}
        \E \[ \( X - \E X \)^2 \]
        \preceq 
        3 nr . 
    \end{align*} 
    

    For $p > 2$, we have 
    \begin{align*} 
        \E \[ X^p \] 
        = 
        n^p \E \[ \( \sum_{i=1}^r u_i^2 \) \u_{:r} \u_{:r}^\top \otimes \v \v^\top \] 
        \preceq 
        \frac{p!}{2} (6 n (r + 2) )^{p-1} I_{n^2} , 
    \end{align*} 
    which, by operator Jensen, implies 
    \begin{align*} 
        \E \[ \( X - \E X \)^p \]
        \preceq 
        \E \[ 2^p X^p + 2^p \(\E X \)^p \] 
        \preceq 
        \frac{p!}{2} (24 n (r + 2) )^{p-1} I_{n^2} . 
    \end{align*} 
    When using the operator Jensen's inequality, we use $ I_{n^2} = \frac{1}{2} I_{n^2} + \frac{1}{2} I_{n^2} $ as the decomposition of identity.  

    Let $ X_1, X_2, \cdots, X_M $ be $iid$ copies of $X$. 
    Since {$ M \ge C nr \log(1/\delta) $}, Theorem \ref{thm:zhu} implies that 
    \begin{align} 
        \Pr \( \left\| \frac{1}{M} \sum_{i=1}^M (Q \otimes Q ) X_i (Q^\top \otimes Q^\top ) - (Q \otimes Q ) \E \[ X \] (Q^\top \otimes Q^\top ) \right\| \ge \frac{1}{6} \) 
        \le {\frac{\delta}{3} }. 
        \label{eq:event-X} 
    \end{align} 

    The bound for $Y$ follows similarly. Let $ Y_1, Y_2, \cdots, Y_M $ be $iid$ copies of $Y$, and we have 
    \begin{align} 
        \Pr \( \left\| \frac{1}{M} \sum_{i=1}^M (Q \otimes Q ) Y_i (Q^\top \otimes Q^\top ) - (Q \otimes Q ) \E \[ Y \] (Q^\top \otimes Q^\top ) \right\| \ge \frac{1}{6} \) 
        \le {\frac{\delta}{3} }. 
        \label{eq:event-Y} 
    \end{align} 
    
    


    \textbf{Bounding $Z$. }
    The second moment of $Z$ is  
    \begin{align*} 
        \E \[ Z^2 \]
        =&\; n^4 \E \[ \( \u_{r+1:} \u_{:r}^\top \otimes \v_{:r} \v^\top + \u_{:r} \u_{r+1:}^\top \otimes \v \v_{:r}^\top   \)^2  \] \\ 
        =& \; 
        n^4 \E \[  \( \u_{r+1:} \u_{:r}^\top \u_{:r} \u_{r+1:}^\top  \) \otimes \( \v_{:r} \v_{:r}^\top \) + \( \u_{:r} \u_{r+1:}^\top \u_{r+1:} \u_{:r}^\top \otimes \v \v_{:r}^\top \v_{:r} \v^\top \) \] \\ 
        \preceq& \; 
        \frac{n^2 r }{ (n+2)} I_{n^2} + 3nr I_{n^2} \preceq 4 nr I_{n^2} , 
    \end{align*} 
    where the last line uses Proposition \ref{prop:p}. 
    
    The $2p$-th power of $Z$ is 
    \begin{align*} 
        Z^{2p} 
        =& \;  
        n^{4p}  \( \u_{r+1:} \u_{:r}^\top \u_{:r} \u_{r+1:}^\top  \)^p \otimes \( \v_{:r} \v_{:r}^\top \)^p \\
        &+ 
        n^{4p}  \( \u_{:r} \u_{r+1:}^\top \u_{r+1:} \u_{:r}^\top \)^p \otimes \( \v \v_{:r}^\top \v_{:r} \v^\top \)^p  \\ 
        \preceq& \; 
        n^{4p} \( \u_{:r}^\top \u_{:r} \)^p \u_{r+1:} \u_{r+1:}^\top \otimes \(  \v_{:r}^\top \v_{:r} \)^{p-1} \v_{:r} \v_{:r}^\top \\
        &+ 
        n^{4p} \( \u_{:r}^\top \u_{:r} \)^{p-1} \u_{:r} \u_{:r}^\top \otimes \( \v_{:r}^\top \v_{:r}  \)^p \v \v^\top 
    \end{align*} 
    and the $(2p+1)$-th power of $Z$ is 
    \begin{align*} 
        Z^{2p+1} 
        =& \;  
        n^{4p+2} \( \u_{r+1:} \u_{:r}^\top \u_{:r} \u_{r+1:}^\top  \)^p \u_{r+1:} \u_{:r}^\top \otimes \( \v_{:r} \v_{:r}^\top \)^p \v_{:r} \v^\top \\ 
        & + \( \u_{:r} \u_{r+1:}^\top \u_{r+1:} \u_{:r}^\top \)^p \u_{:r} \u_{r+1:}^\top \otimes \( \v \v_{:r}^\top \v_{:r} \v^\top \)^p \v \v_{:r}^\top . 
    \end{align*} 
    
    Thus by Proposition \ref{prop:p}, we have 
    \begin{align*} 
        \E \[ Z^{2p} \] 
        \preceq& \;  
        n^{4p} \E \[ r^{p-1} \( \sum_{i=1}^r u_i^{2p} \) \u_{r+1:} \u_{r+1:}^\top \otimes r^{p-2} \( \sum_{i=1}^r v_i^{2p-2} \) \v \v^\top \] \\ 
        &+ 
        n^{4p} \E \[ r^{p-2} \( \sum_{i=1}^r u_i^{2p-2} \) \u_{:r} \u_{:r}^\top \otimes r^{p-1} \( \sum_{i=1}^r v_i^{2p} \) \v \v^\top \] \\ 
        \preceq& \; 
        2 n^{4p} r^{2p-1} \cdot \frac{ (2p+1) (2p-1) \cdots 1 }{ n (n+2) \cdots (n+2p) } \cdot \frac{ (2p-1) (2p-3) \cdots 1 }{ n (n+2) \cdots (n+2p - 2) } I_{n^2} \\ 
        \preceq& \;  
        \frac{(2p)!}{2} (8 nr)^{2p-1} I_{n^2} . 
    \end{align*} 
    For $Z^{2p+1}$ ($p \in \mathbb{N}$), we notice that 
    \begin{align*}
        \E \[ \( \u_{r+1:} \u_{:r}^\top \u_{:r} \u_{r+1:}^\top  \)^p \u_{r+1:} \u_{:r}^\top \] = \E \[ \( \u_{:r} \u_{r+1:}^\top \u_{r+1:} \u_{:r}^\top \)^p \u_{:r} \u_{r+1:}^\top \] = 0, 
    \end{align*}
    since these terms only involve odd powers of the entries of $\u$. Therefore 
    \begin{align} 
        \E \[ Z^{2p+1} \] = 0, \quad \text{ for } p = 0,1,2,\cdots \label{eq:Z-odd} 
    \end{align} 
    Let $Z_1, Z_2, \cdots, Z_M$ be $M$ $iid$ copies of $Z$, and {$M \ge C nr \log (1/\delta) $} for some absolute constant $C$. 
    By (\ref{eq:Z-odd}), we know $ \E \[ Z\] = 0 $, and all the above moments of $Z$ are centralized moments of $Z$. 
    Now we apply Theorem \ref{thm:zhu} to conclude that: 
    \begin{align} 
        & \; \Pr \( \left\| \frac{1}{M} \sum_{i=1}^M Z_i - \E Z \right\| \ge \frac{1}{6} \) \nonumber \\ 
        =& \;  
        \Pr \( \left\| \frac{1}{M} \sum_{i=1}^M (Q \otimes Q ) Z_i (Q^\top \otimes Q^\top ) - (Q \otimes Q ) \E \[ Z \] (Q^\top \otimes Q^\top ) \right\| \ge \frac{1}{6} \) \nonumber \\ 
        =& \;  
        \Pr \( \left\| \frac{1}{M} \sum_{i=1}^M (Q^\top \otimes Q^\top ) Z_i (Q^\top \otimes Q^\top ) \right\| \ge \frac{1}{6} \) 
        \le {\frac{\delta}{3} }, 
        \label{eq:event-Z} 
    \end{align} 
    where $Q$ is the orthogonal matrix as introduced in (\ref{eq:equi-dist}). 
    We take a union bound over (\ref{eq:event-X}), (\ref{eq:event-Y}) and (\ref{eq:event-Z}) to conclude the proof. 


\end{proof} 

Now with Lemma \ref{lem:op-concen} in place, we state next Lemma \ref{lem:infeas-2}. This lemma proves \textbf{(A1)}. 

\begin{lemma} 
    \label{lem:infeas-2} 
    Suppose $\mathcal{E}_1$ is true. Let $ \wh{H} $ be the solution of the constrained optimization problem, and let $\Delta := \wh{H} - H$. Then $ \| \P_{T} \Delta \|_2 \le 2n \| \P_{T}^\perp \Delta \|_2 $. 
\end{lemma} 

\begin{proof} 

    
    Represent $\S$ as a matrix of size $ n^2 \times n^2 $. Let $ \sqrt{\S} $ be defined as a canonical matrix function. That is, $ \sqrt{\S} $ and $ \S $ share the same eigenvectors, and the eigenvalues of $ \sqrt{\S} $ are the square roots of the eigenvalues of $\S$. 
    Clearly, 
    \begin{align} 
        \| \sqrt{\S} \Delta \|_2 
        = 
        \| \sqrt{\S} \P_{T}^\perp \Delta + \sqrt{\S} \P_{T} \Delta \|_2 
        \ge 
        \| \sqrt{\S} \P_{T} \Delta \|_2 - \| \sqrt{\S} \P_{T}^\perp \Delta \|_2 . \label{eq:infeasible-2-3} 
    \end{align} 
    
    Clearly we have 
    \begin{align*} 
        \| \sqrt{ \S } \P_{T}^\perp \Delta \|_2 
        \le  
        n \| \P_{T}^\perp \Delta \|_2 . 
    \end{align*} 

    Also, it holds that 
    \begin{align} 
        & \; \| \sqrt{\S} \P_{T} \Delta \|_2^2 
        =  
        \< \sqrt{\S} \P_{T} \Delta, \sqrt{\S} \P_{T} \Delta \>
        = 
        \< \P_{T} \Delta, \P_{T} \S \P_{T} \Delta \> \nonumber \\ 
        =& \; 
        \| \P_{T} \Delta \|_2^2 - \< \P_{T} \Delta - \P_{T} \S \P_{T} \Delta , \P_{T} \Delta \> 
        \ge 
        \frac{1}{2} \| \P_{T} \Delta \|_2^2 , \label{eq:infeasible-2-4}
    \end{align} 
    where the last inequality uses Lemma \ref{lem:op-concen}. 

    Since $ \wh{H} $ solves (\ref{eq:goal}), we know $ \S \Delta = 0 $, and thus $ \sqrt{\S} \Delta = 0 $. 
    Suppose, in order to get a contradiction, that $ \| \P_{T} \Delta \|_2 > 2n \| \P_{T}^\perp \Delta \|_2 $. 
    Then (\ref{eq:infeasible-2-3}) and (\ref{eq:infeasible-2-4}) yield  
    \begin{align*} 
        \| \sqrt{\S} \Delta \|_2 
        \ge 
        \frac{1}{2} \| \P_{T} \Delta \|_2 - n \| \P_{T}^\perp \Delta \|_2 > 0, 
    \end{align*} 
    which leads to a contraction. 
\end{proof}  



Next we turn to prove \textbf{(A2)}, whose core argument relies on Lemma \ref{lem:mat-concen}. 

\begin{lemma} 
    \label{lem:mat-concen}
    Let $G \in T$ be fixed. Pick any $\delta \in (0,1)$. Then there exists a constant $C$, such that when $ M \ge C n r^2 \log (1/\delta) $, it holds that 
    \begin{align*} 
        \Pr \( \left\| \P_T^\perp \S G \right\| \ge \frac{1}{4 \sqrt{r}} \| G \| \) \le \delta . 
    \end{align*} 
\end{lemma}

\begin{proof} 

    There exists an orthogonal matrix $Q$, such that $ G = Q \Lambda Q^\top $, where 
    \begin{align*}
        \Lambda = \textrm{Diag} ( \lambda_1, \lambda_2, \cdots , \lambda_{2r}, 0,0,\cdots, 0 )
    \end{align*}
    is a diagonal matrix consists of eigenvalues of $G$. Let $ \P $ be the operator defined as in (\ref{eq:def-P}), and we will study the behavior of $ \P G $ and then apply Theorem \ref{thm:zhu}. 
    Since the distribution of $\u, \v \sim \text{Unif} (\mathbb{S}^{n-1})$ is rotation-invariant and reflection-invariant, 
    we have 
    \begin{align*} 
        \P G
        = 
        n^2 \u \u^\top G \v \v^\top 
        =_d
        n^2 Q \u \u^\top Q^\top G Q \v \v^\top Q^\top
        =
        n^2 Q \u \u^\top \Lambda \v \v^\top Q^\top , 
    \end{align*} 
    where $ =_d $ denotes distributional equivalence. 
    Thus it suffices to study the behavior of $B := n^2 Q \u \u^\top \Lambda \v \v^\top Q^\top $. For the matrix $B$, we consider 
    \begin{align*}
        A := 
        \begin{bmatrix}
            0_{n \times n} & B \\ 
            B^\top & 0_{n \times n} 
        \end{bmatrix} . 
    \end{align*}
    
    
    Next we study the moments of $A$. 
    The second power of $A$ is $ A^2 = 
        \begin{bmatrix} 
            B B^\top & 0_{n \times n} \\ 
            0_{n \times n} & B^\top B 
        \end{bmatrix} . $ 
    By Proposition \ref{prop:p}, we have 
    \begin{align*} 
        \E \[ B B^\top \]
        =& \; 
        n^4 \E \[ Q \u \u^\top \Lambda \v \v^\top \Lambda \u \u^\top Q^\top \] \\ 
        =& \;  
        n^3 Q \E \[ \u \u^\top \Lambda^2 \u \u^\top \] Q^\top \\
        \preceq& \;  
        n^3 Q \E \[ \| G \|^2 \u \(\u_{:2r}^\top \u_{:2r}\) \u^\top \] Q^\top 
        \preceq 
        4 nr \| G \|^2 I_{n} , 
    \end{align*}
    and similarly, $ \E \[ B^\top B \] \preceq 4 nr \| G \|^2 I_{n} $. For even moments of $A$, we first compute $ \E \[ \( B B^\top \)^p \] $ and $ \E \[ \( B^\top B \)^p \] $ for $p \ge 2$. For this, we have 
    \begin{align} 
        & \; \E \[ \(B^\top B \)^p \] 
        = 
        Q \E \[ n^{4p} \( \sum_{i=1}^{2r} \lambda_i v_i u_i \)^{2p} \v \v^\top \] Q^\top \nonumber \\ 
        =& \; 
        n^{4p} Q \E \[ \(  \sum_{ \substack{\alpha_1, \alpha_2, \cdots, \alpha_{2r} \ge 0;\\ \sum_{i=1}^{2r} \alpha_i = 2p }} { 2p \choose \alpha_1, \alpha_2, \cdots, \alpha_{2r} } \prod_{i=1}^{2r} \( \lambda_i v_i u_i \)^{\alpha_i} \) \v \v^\top \] Q^\top \nonumber \\ 
        =& \; 
        n^{4p} Q \E \[ \( \sum_{ \substack{\alpha_1, \alpha_2, \cdots, \alpha_{2r} \ge 0; \\ \sum_{i=1}^{2r} \alpha_i = 2p ; \; \alpha_i \text{ even} } } { 2p \choose \alpha_1, \alpha_2, \cdots, \alpha_{2r} } \prod_{i=1}^{2r} \( \lambda_i v_i u_i \)^{\alpha_i} \) \v \v^\top \] Q^\top , \label{eq:for-mom-1}
    \end{align} 
    where the last inequality uses that expectation of odd powers of $v_i$ or $u_i$ are zero. Note that 
    \begin{align} 
        & \; \sum_{ \substack{\alpha_1, \alpha_2, \cdots, \alpha_{2r} \ge 0; \\ \sum_{i=1}^{2r} \alpha_i = 2p ; \; \alpha_i \text{ even} } } { 2p \choose \alpha_1, \alpha_2, \cdots, \alpha_{2r} } \prod_{i=1}^{2r} \( \lambda_i v_i u_i \)^{\alpha_i} \nonumber \\ 
        =& \;  
        \sum_{\substack{\alpha_1, \alpha_2, \cdots, \alpha_{2r} \ge 0; \\ \sum_{i=1}^{2r} \alpha_i = 2p ; \; \alpha_i \text{ even} } } \frac{(2p)!}{p!} \prod_{i=1}^{2r} \frac{(\frac{\alpha_i}{2})!}{\alpha_i!} { p \choose \frac{\alpha_1}{2}, \frac{\alpha_2}{2}, \cdots, \frac{\alpha_{2r}}{2} } \prod_{i=1}^{2r} \( \lambda_i v_i u_i \)^{\alpha_i} \nonumber \\ 
        \le& \; 
        (200 r)^{p-1} \sum_{\alpha_1, \alpha_2, \cdots, \alpha_{2r} \ge 0; \sum_{i=1}^{2r} \alpha_i = p } { p \choose {\alpha_1}, {\alpha_2}, \cdots, {\alpha_{2r} } } \prod_{i=1}^{2r} \( \lambda_i^2 v_i^2 u_i^2 \)^{\alpha_i} \nonumber \\
        =& \; 
        (200 r)^{p-1} \( \sum_{i=1}^{2r} \lambda_i^2 u_i^2 v_i^2 \)^p , \label{eq:for-mom-2}
    \end{align} 
    where the inequality on the last line uses Lemma \ref{lem:wu}. Now we combine (\ref{eq:for-mom-1}) and (\ref{eq:for-mom-2}) to obtain 
    \begin{align}
        & \; \E \[ (B^\top B)^p \] 
        \preceq 
        n^{4p} (200 r)^{p-1} Q \E \[ \( \sum_{i=1}^{2r} \lambda_i^2 u_i^2 v_i^2 \)^p \v \v^\top \] Q^\top \\
        \preceq&\;  
        n^{4p} (200 r)^{2p-2} Q \E \[ \( \sum_{i=1}^{2r} \lambda_i^{2p} u_i^{2p} v_i^{2p} \) \v \v^\top \] Q^\top \nonumber \\ 
        \preceq & \; 
        \frac{(2p)!}{2} \max_{ i } \lambda_i^{2p} ( C n r )^{2p-1} I_n = \frac{(2p)!}{2} \| G \|^{2p} ( C n r )^{2p-1} I_n, \label{eq:mom-B-even} 
    \end{align} 
    where the inequality on the last line uses Proposition \ref{prop:p}. Similarly, we have 
    \begin{align*}
        \E \[ ( B B^\top)^p \] \preceq \frac{(2p)!}{2} \| G \|^{2p} ( 200n r )^{2p-1} I_n . 
    \end{align*}
    
    
    Therefore, we have obtained a bound on even moments of $A$: 
    \begin{align*}
        \E \[ A^{2p} \] 
        = 
        \begin{bmatrix}
            \E \[ \( B B^\top \)^p \] & 0_{n \times n} \\ 
            0_{n \times n} & \E \[ \( B^\top B \)^p \]  
        \end{bmatrix} 
        \preceq
        \frac{(2p)!}{2} \| G \|^{2p} ( 200n r )^{2p-1} I_{2n} , 
    \end{align*}
    for $ p =2,3,4,\cdots $, 
    and thus a bound on the centralized moments on even moments of $A$: 
    \begin{align*}
        \E \[ \( A - \E A \)^{2p} \] 
        \preceq
        \frac{(2p)!}{2} \| G \|^{2p} ( 400n r )^{2p-1} I_{2n} , \quad p =2,3,4,\cdots 
    \end{align*} 
    Next we upper bound the odd moments of $ A $. Since  
    \begin{align*} 
        \E \[ A ^{2p+1} \] 
        = 
        \begin{bmatrix} 
            0_{n \times n} & \E \[ \( B B^\top \)^p B \]  \\ 
            \E \[ \( B^\top B \)^p B^\top \] & 0_{n \times n} 
        \end{bmatrix} , 
    \end{align*} 
    it suffices to study $ \E \[ \( B B^\top \)^p B \] $ and $ \E \[ \( B^\top B \)^p B^\top \] $.  
    Since 
    \begin{align*} 
        \( B B^\top \)^p B 
        = 
        n^{4p+2} \( \sum_{i=1}^{2r} \lambda_i v_i u_i \)^{2p} Q \v \v^\top \Lambda \u \u^\top Q^\top , 
    \end{align*} 
    using the arguments leading to (\ref{eq:mom-B-even}), we have 
    \begin{align} 
        &\E \[ \( B B^\top \)^p B \] \preceq \frac{(2p+1)!}{2} (Cnr)^{2p} \| G \|^{2p+1} I_n , \nonumber \\ 
        &\E \[ \( B^\top B \)^p B^\top \] \preceq \frac{(2p+1)!}{2} (Cnr)^{2p} \| G \|^{2p+1} I_n . \label{eq:lem4-2}
    \end{align} 
    Since $ \begin{bmatrix}
        0_{n \times n} & I_n \\ I_n & 0_{n \times n}
    \end{bmatrix} \preceq 2 I_{2n} $, the above two inequalities in (\ref{eq:lem4-2}) implies 
    \begin{align*} 
        \E \[ A^{2p+1} \] 
        \preceq 
        \frac{(2p+1)!}{2} (Cnr)^{2p} \| G \|^{2p+1} I_{2n} , 
    \end{align*} 
    and thus 
    \begin{align*}
        \E \[ \(A - \E A \)^{2p+1} \] 
        \preceq 
        \frac{(2p+1)!}{2} (Cnr)^{2p} \| G \|^{2p+1} I_{2n}. 
    \end{align*}
    Now we have established moment bounds for $ A $, thus also for $ \P_{T}^\perp \P G $. From here we apply Theorem \ref{thm:zhu} to conclude the proof. 

\end{proof}

The next lemma will essentially establish \textbf{(A2)}. This argument relies on the existence of a dual certificate \citep{candes2010power,gross2011recovering,candes2012exact}. 

\begin{lemma}
    \label{lem:exist-cert-2}
    Pick $\delta > 0$. 
    Define 
    \begin{align*}
        \mathcal{E}_2 := \left\{ \exists \; Y \in range(\S) : \| \P_{T} Y - \sign (H) \|_2 \le \frac{1}{ n^{21} } 
        \quad \text{and} \quad 
        \| \P_{T}^\perp Y \| \le \frac{1}{2} \right\} . 
    \end{align*}
    Let $L = 12 \log_2 n$. 
    Let $ m \ge c \cdot n r^2 \log \( \frac{L}{\delta} \) $ for some constant $c$. If $M = m L \ge c \cdot n r^2 \log n \log \( \frac{\log n}{\delta} \)$ for some constant $c$, then $ \Pr \( \mathcal{E}_2 \) \ge 1 - \delta $. 
\end{lemma} 


\begin{proof} 
    Following \citep{gross2011recovering}, we define random projectors $ \wt{\S}_l $ ($1 \le l \le L$), such that 
    \begin{align*} 
        \wt{\S}_l := \frac{1}{m} \sum_{j=1}^m \P_{m(l-1) + j} .  
    \end{align*} 

    Then define 
    \begin{align*} 
        X_0 = \sign (H), \quad Y_i = \sum_{j=1}^i \wt{\S}_j \P_{T} X_{j-1}, \quad X_i = \sign (H) - \P_{T} Y_i , 
        \quad \forall i \ge 1. 
    \end{align*} 
    
    From the above definition, we have 
    \begin{align*}  
        X_i = ( \P_{T} - \P_{T} \wt{\S}_i \P_{T} ) ( \P_{T} - \P_{T} \wt{\S}_{i-1} \P_{T} ) \cdots ( \P_{T} - \P_{T} \wt{\S}_1 \P_{T} ) X_0 , \quad \forall i \ge 1. 
    \end{align*} 

    Now we apply Lemma \ref{lem:op-concen} to $ \wt{\S}_1, \wt{\S}_2, \cdots , \wt{\S}_L $, and get, when event $\mathcal{E}_1$ is true for all $\wt{\S}_i, i = 1,2,\cdots,L$, 
    \begin{align} 
        \| X_{i} \|_2 \le \frac{1}{4} \| X_{i-1} \|_2 \le \cdots 
        \le 
        \frac{\sqrt{r}}{4^i} , \quad \forall i = 1,2,\cdots,L
        \label{eq:cert-1} 
    \end{align} 
    Note that with probability exceeding $1 - \frac{\delta}{2}$, $\mathcal{E}_1$ is true for all $\wt{\S}_i$, $i = 1,2,\cdots, L$. 
    Since $\wt{\S}_{i}$ are mutually independent, $\wt{\S}_{i+1}$ is independent of $X_i$ for each $i \in \{ 0, 1,\cdots, L-1 \}$. In view of this, we can apply Lemma \ref{lem:mat-concen} to $\P_{T}^\perp Y_{L}$ followed by a union bound, and get, with probability exceeding $1 - \frac{\delta}{2}$, 
    \begin{align} 
        \| \P_{T}^\perp Y_{L} \| 
        \le 
        \sum_{i=1}^L \frac{ 1 }{4\sqrt{r}} \| X_{i-1} \|_2 
        \le 
        \frac{ 1 }{4 } \sum_{i=1}^L \frac{1}{4^{i-1}}  
        \le 
        \frac{1}{2} . \label{eq:cert-2}
    \end{align} 

    Now combining (\ref{eq:cert-1}) and (\ref{eq:cert-2}) finishes the proof. 
    
    
\end{proof} 


Now we are ready to prove Theorem \ref{thm:main}. 

\begin{proof}[Proof of Theorem \ref{thm:main}] 
    Let $\mathcal{E}_2$ be true. 
    Then there exists $Y$ such that $\< Y ,\Delta \> = 0$, since $\S \Delta = 0$. Thus we have
    \begin{align*}
        & \; \< \sign(H),  P_{U} \Delta P_{U} \> 
        = 
        \< \sign(H),  \Delta \> 
        = 
        \< \sign(H) - Y,  \Delta \> \\ 
        =& \;  
        \< \P_T \( \sign(H) - Y \) ,  \Delta_T \> + \< \P_T^\perp \( \sign(H) - Y \) ,  \Delta_T^\perp \> \\ 
        =& \; 
        \< \sign(H) - \P_T Y ,  \Delta_T \> - \< \P_T^\perp  Y ,  \Delta_T^\perp \> \\
        \ge& \;  
        - \frac{1}{n^{21}} \| \Delta_T \|_2 - \frac{1}{2} \| \Delta_T^\perp \|_1 
        , 
    \end{align*}
    where the last inequality uses Lemma \ref{lem:exist-cert-2}. 

    Now, by Lemma \ref{lem:prepare} and Lemma \ref{lem:infeas-2}, we have 
    \begin{align*}
        0 
        \ge 
        \frac{1}{2} \| \Delta_T^\perp \|_1 - \frac{1}{n^{21}} \| \Delta_T \|_2 
        \ge 
        \frac{1}{2} \| \Delta_T^\perp \|_1 - \frac{1}{n^{20}} \| \Delta_T \|_1 
        \ge 
        \frac{1}{2} \| \Delta_T^\perp \|_1 - \frac{2}{n^{18}} \| \Delta_T^\perp \|_1 , 
    \end{align*}
    which implies $ \| \Delta_T^\perp \|_1 = 0 $. Finally another use of Lemma \ref{lem:infeas-2} implies $ \| \Delta \|_1 = 0 $, which concludes the proof. 
    
\end{proof}

Theorem \ref{thm:main}, together with Proposition \ref{prop:measure}, establishes Corollary \ref{cor:hess}.

\section{Conclusion} 

In this paper, we consider the Hessian estimator problem via matrix recovery techniques. In particular, we show that the finite-difference method studied in \citep{10.1093/imaiai/iaad014,wang2022hess}, together with a convex program, guarantees a high probability recovery of a rank-$r$ Hessian using $ n r^2 $ (up to logarithmic and constant factors) finite-difference operations. Compared to matrix completion methods, we do not assume any incoherence between the coordinate system and the hidden singular space of the Hessian matrix. 
In a follow-up work, we apply the Hessian estimation mechanism to Newton's cubic method
\citep{nesterov2006cubic,Nesterov2008}, and design sample-efficient optimization algorithms for functions with (approximately) low-rank Hessian.

\section*{Acknowledgement}

The authors thank Dr. Hehui Wu for insightful discussions and his contributions to Lemma \ref{lem:wu} and Dr. Abiy Tasissa for helpful discussions.





\bibliographystyle{apalike} 
\bibliography{references} 

\appendix
\section{Auxiliary Propositions and Lemmas} 

\begin{proof}[Proof of Proposition \ref{prop:p}]
    Let $ (r, \varphi_1, \varphi_2, \cdots, \varphi_{n-1}) $ be the spherical coordinate system. We have, for any $i = 1,2,\cdots, n$ and an even integer $p$, 
    \begin{align*} 
        \E \[ v_1^p \] 
        =& \; 
        \frac{1}{A_n} \int_{0}^{2\pi} \int_0^\pi \cdots \int_{0}^\pi \cos^p ( \varphi_1 ) \sin ^{n-2}(\varphi _{1})\sin ^{n-3}(\varphi _{2})\cdots \sin(\varphi _{n-2}) \,d\varphi _{1}\, d\varphi_{2} \cdots d \varphi_{n-1} , 
    \end{align*} 
    where $A_n$ is the surface area of $ \mathbb{S}^{n-1} $. 
    Let 
    \begin{align*}
        I (n,p) 
        := 
        \int_{0}^\pi \sin^n ( x ) \cos^p (x) \, dx. 
    \end{align*} 
    
    Clearly, $ I (n,p) = I (n, p - 2 ) - I ( n + 2, p - 2 ) $. By integration by parts, we have $ I ( n + 2, p - 2 ) = \frac{n+1}{p-1} I (n,p) $. The above two equations give $ I (n,p) = \frac{p-1}{n+p} I (n, p-2) $. 
    
    Thus we have $ \E \[ v_1^p \] = \frac{ I (n-2, p) }{ I ( n - 2, 0 ) } = \frac{ I (n-2, p) }{ I ( n - 2, p-2 ) } \frac{ I (n-2, p-2) }{ I ( n - 2, p-4 ) }  \cdots \frac{ I (n-2, 2) }{ I ( n - 2, 0 ) } = \frac{ (p-1) (p-3) \cdots 1 }{ n (n+2) \cdots (n+p - 2) } $. 
    We conclude the proof by symmetry. 
    
\end{proof}

\end{document}